\documentclass[11pt,letterpaper]{article}

\usepackage[margin=1in]{geometry}
\usepackage{amsmath}
\usepackage{amssymb}
\usepackage{amsthm}
\usepackage{multirow}
\usepackage{enumerate}
\usepackage{longtable}
\usepackage{caption2}

\allowdisplaybreaks[1]

\newtheorem{theorem}{Theorem}

\newtheorem{lemma}[theorem]{Lemma}
\newtheorem{proposition}[theorem]{Proposition}
\newtheorem{definition}[theorem]{Definition}

\theoremstyle{definition}

\newtheorem{remark}[theorem]{Remark}

\numberwithin{theorem}{section}
\numberwithin{equation}{section}
\numberwithin{table}{section}

\newcommand{\bb}{\mathbb}
\newcommand{\Z}{\bb{Z}}

\newcommand{\R}{\bb{R}}

\newcommand{\cC}{\mathcal{C}}

\newcommand{\cP}{\mathcal{P}}
\newcommand{\cQ}{\mathcal{Q}}

\newcommand{\de}{\delta}
\newcommand{\De}{\Delta}

\newcommand{\Ga}{\Gamma}

\newcommand{\La}{\Lambda}

\newcommand{\var}{\varphi}
\newcommand{\Sig}{\Sigma}
\newcommand{\sig}{\sigma}

\newcommand{\pr}{\prime}

\newcommand{\sm}{\setminus}

\newcommand{\lan}{\langle}
\newcommand{\ran}{\rangle}

\newcommand{\wh}{\widehat}

\newcommand{\Aut}{\text{Aut}}

\newcommand{\Sym}{\text{Sym}}

\newcommand{\covol}{\text{covol}}
\newcommand{\vol}{\text{vol}}

\long\def\symbolfootnote[#1]#2{\begingroup%
\def\thefootnote{\fnsymbol{footnote}}\footnote[#1]{#2}\endgroup}

\begin{document}

\title{A Direct Construction of Primitive Formally Dual Pairs Having Subsets with Unequal Sizes}
\author{Shuxing Li \and Alexander Pott}
\date{}
\maketitle

\symbolfootnote[0]{
S.~Li and A.~Pott are with the Faculty of Mathematics, Otto von Guericke University Magdeburg, 39106 Magdeburg, Germany (e-mail: shuxing.li@ovgu.de, alexander.pott@ovgu.de).
}

\begin{abstract}
The concept of formal duality was proposed by Cohn, Kumar and Sch\"urmann, which reflects a remarkable symmetry among energy-minimizing periodic configurations. This formal duality was later translated into a purely combinatorial property by Cohn, Kumar, Reiher and Sch\"urmann, where the corresponding combinatorial objects were called formally dual pairs. So far, except the results presented in \cite{LP}, we have little information about primitive formally dual pairs having subsets with unequal sizes. In this paper, we propose a direct construction of primitive formally dual pairs having subsets with unequal sizes in $\Z_2 \times \Z_4^{2m}$, where $m \ge 1$. This construction recovers an infinite family obtained in \cite{LP}, which was derived by employing a recursive approach. Although the resulting infinite family was known before, the idea of the direct construction is new and provides more insights which were not known from the recursive approach.


\end{abstract}

\section{Introduction}

Let $\cC$ be a particle configuration in the Euclidean space $\R^n$. Let $f: \R^n \rightarrow \R$ be a potential function, which is used to measure the energy possessed by $\cC$. The energy minimization problem aims to find configurations $\cC \subset \R^n$ with a fixed density, whose energy is minimal with respect to a potential function $f$. In physics, the energy minimization problem amounts to find the ground states in a given space, with respect to a prescribed density and potential function. This problem is of great interest and notoriously difficult in general \cite[Section I]{CKS}. For instance, the famous sphere packing problem can be viewed as an extremal case of the energy minimization problem \cite[p. 123]{CKRS}.

In 2009, Cohn, Kumar and Sch\"{u}rmann considered a weaker version of the energy minimization problem, where the configurations under consideration are restricted to so called periodic configurations \cite{CKS}. A periodic configuration is formed by a union of finitely many translations of a lattice. For instance, let $\La$ be a lattice in $\R^n$, then $\cP=\bigcup_{i=1}^N (v_i+\La)$ is a periodic configuration formed by $N$ translations of $\La$. The \emph{density} of $\cP$ is defined to be $\de(\cP)=N/\covol(\La)$, where $\covol(\La)=\vol(\R^n/\La)$ is the volume of a fundamental domain of $\La$. Given a potential function $f: \R^n \rightarrow \R$, define its Fourier transformation
$$
\wh{f}(y)=\int_{\R^n}f(x)e^{-2\pi i\lan x,y \ran}dx,
$$
where $\lan \cdot, \cdot \ran$ is the inner product in $\R^n$. The potential functions belong to the class of Schwartz function, so that their Fourier transformations are well-defined. For a Schwartz function $f: \R^n \rightarrow \R$ and a periodic configuration $\cP=\bigcup_{j=1}^N(v_j+\La)$ associated with a lattice $\La \subset \R^n$, define the \emph{average pair sum} of $f$ over $\cP$ as
$$
\Sig_f(\cP)=\frac{1}{N}\sum_{j,\ell=1}^N\sum_{x \in \La}f(x+v_j-v_\ell),
$$
which is used to measure the energy possessed by the periodic configuration $\cC$ with respect to the potential function $f$. Given a density $0<\de<1$ and a Schwartz potential function $f$, the energy minimization problem concerning periodic configurations aims to find periodic configurations $\cP$ so that $\Sig_f(\cP)$ is minimal and $\de(\cP)=\de$.

Based on numerical experiments, Cohn et al. observed that each energy-minimizing periodic configuration obtained in their simulations possesses a remarkable symmetry called formal duality \cite[Section VI]{CKS}. More precisely, if $\cP$ is an energy-minimizing periodic configuration, then numerous experiments suggested that there exists a periodic configuration $\cQ$, so that for \emph{each} Schwartz function $f$, we have
\begin{equation}\label{def-formaldual}
\Sig_f(\cP)=\de(\cP)\Sig_{\wh{f}}(\cQ).
\end{equation}
If two periodic configurations $\cP$ and $\cQ$ satisfy \eqref{def-formaldual} for each Schwartz function $f$, then they are called formally dual to each other \cite[Definition 2.1]{CKRS}. This formal duality among periodic configurations revealed a deep symmetry which has not been well understood.

Remarkably, Cohn, Kumar, Reiher and Sch\"{u}rmann realized that formal duality among a pair of periodic configurations can be translated into a purely combinatorial property \cite[Theorem 2.8]{CKRS}. Indeed, they introduced the concept of formally dual pairs in finite abelian groups, which is a combinatorial counterpart of formal duality \cite[Definition 2.9]{CKRS}. Let $\La \subset \R^n$ be a lattice with a basis containing $n$ vectors. The dual lattice of $\La$ is defined as
$$
\La^*=\{x \in \R^n \mid \lan x, y \ran \in \Z, \forall y \in \La \},
$$
in which $\lan \cdot, \cdot \ran$ is the inner product in $\R^n$. Let $\cP=\bigcup_{j=1}^N(v_j+\La)$ and $\cQ=\bigcup_{j=1}^M(w_j+\Ga)$ be two periodic configurations. Define $\cP-\cP$ to be the subset $\{ x-y \mid x,y \in \cP \}$. Suppose $\cP-\cP \subset \Ga^*$ and $\cQ-\cQ \subset \La^*$. Then, as observed in \cite[p. 129]{CKRS}, the two quotient groups $\Ga^*/\La$ and $\La^*/\Ga$ satisfy that $\Ga^*/\La \cong \La^*/\Ga \cong G$, where $G$ is a finite abelian group. Moreover, the two sets $S=\{v_j \mid 1 \le j \le N\}$ and $T=\{w_j \mid 1 \le j \le M\}$ can be regarded as subsets of $G$, so that $S$ corresponds to $\cP$ and $T$ corresponds to $\cQ$. Cohn et al.'s key observation was that, $\cP$ and $\cQ$ are formally dual if and only if $S$ and $T$ form a formally dual pair in $G$ (see Definition~\ref{def-iso} for the concept of formally dual pairs).
Consequently, the formal duality among periodic configurations $\cP$ and $\cQ$ was reduced to the property of a pair of subsets $S$ and $T$ in a finite abelian group $G$.

Hence, Cohn et al.'s results paved the way of applying combinatorial approach to deal with energy-minimizing periodic configurations. On one hand, let $S=\{v_j \mid 1 \le j \le N\}$ and $T=\{w_j \mid 1 \le j \le M\}$ be a formally dual pair in a finite abelian group $G$. Then for each pair of lattices $\La$ and $\Ga$, satisfying $\Ga^*/\La \cong \La^*/\Ga \cong G$, we have that $\cP=\bigcup_{j=1}^N (v_j+\La)$ and $\cQ=\bigcup_{j=1}^M (w_j+\Ga)$ are formally dual periodic configurations. Hence, given a formally dual pair $S$ and $T$ in $G$, by choosing proper underlying lattices $\La$ and $\Ga$, we can derive infinitely many formally dual periodic configurations $\cP$ and $\cQ$, which are natural candidates of energy-minimizing periodic configurations. On the other hand, let $\La$ and $\Ga$ be two lattices satisfying that $\Ga^*/\La \cong \La^*/\Ga \cong G$, where $G$ is a finite abelian group. Let $\cP$ be a periodic configuration associated with the lattice $\La$ and $\cQ$ be a periodic configuration associated with the lattice $\Ga$, such that $\cP-\cP \subset \Ga^*$ and $\cQ-\cQ \subset \La^*$. The nonexistence of formally dual pairs in $G$ implies that no matter how the periodical configurations $\cP$ and $\cQ$ are formed by taking the union of cosets of $\La$ and $\Ga$, they can never be formally dual. Hence, the nonexistence of formally dual pairs in one finite abelian group $G$ rules out infinitely many potential pairs of formally dual periodic configurations. In a word, formally dual pairs capture the essential information of formally dual periodic configurations, and therefore, offers an elegant combinatorial way to study the formal duality of periodic configurations.

Now we give a brief summary of known results about formally dual pairs. The pioneering works \cite{CKRS,CKS} included some fundamental results and proposed a main conjecture \cite[p. 135]{CKRS}, stating that there are no primitive formally dual pairs in cyclic groups, except two small examples (see Definition~\ref{def-primi} for the concept of primitive formally dual pairs). Motivated by this conjecture, some follow-up works studied formally dual pairs in cyclic groups. Specifically, this conjecture was proved for cyclic groups of prime power order, where Sch\"uler confirmed the odd prime power case \cite{Schueller} and Xia confirmed the even prime power case \cite{Xia}. Malikiosis showed that the conjecture holds true in many cases when the order of the cyclic group is a product of two prime powers \cite{Mali}. Remarkably, his results employed the field descent method, a deep number theoretical approach which has been used to achieve significant progress in the Barker sequence conjecture \cite{LS,Sch}. In \cite[Section 4.2]{LPS}, the authors proposed a new viewpoint towards the conjecture, by building a connection between the two known examples of primitive formally dual pairs in cyclic groups and cyclic relative difference sets.

While there seem to be very few formally dual pairs in cyclic groups, it is natural to ask what is the situation for finite abelian groups. A systematic study of formally dual pairs in finite abelian groups was presented in \cite{LPS}, which contains constructions, classifications, nonexistence results and enumerations. In particular, the first example of primitive formally dual pairs having subsets with unequal sizes was discovered in \cite[Example 3.22]{LPS}, which belongs to the group $\Z_2\times\Z_4^2$. Motivated by this example, the authors constructed many infinite families of such primitive formally dual pairs in \cite{LP}. Indeed, for $m \ge 2$, the authors obtained $m+1$ pairwise inequivalent primitive formally dual pairs in $\Z_2 \times \Z_4^{2m}$, which have subsets with unequal sizes (see Definition~\ref{def-primi} for the concept of inequivalence).

In \cite[Theorem 6.2]{LP}, the authors presented an infinite family of primitive formally dual pair having subsets with unequal sizes. More precisely, the authors used a recursive approach to generate a primitive formally dual pair $S$ and $T$ in $\Z_2 \times \Z_4^{2m}$, $m \ge 1$, such that $|S|=2^{2m}$ and $|T|=2^{2m+1}$. Instead, in this paper, we give a direct construction which exactly recovers this family. This direct construction offers more insights into the construction of primitive formally dual pair having subsets with unequal sizes, which suggests the possibility of more direct constructions. Moreover, it reveals more detailed information about this family, so that the difference spectrum of $T$ can be determined (see the paragraph after Definition~\ref{def-equiv} for the concept of difference spectrum).

The rest of the paper is organized as follows. In Section~\ref{sec2}, we give a brief introduction to formally dual pairs and decribe a lifting construction framework producing new primitive formally dual pairs from known ones. Applying this framework in Section~\ref{sec3}, we present a direct construction of primitive formally dual pairs in $\Z_2 \times \Z_4^{2m}$, which reproduces the infinite family presented in \cite[Theorem 6.2]{LP} and reveals more detailed information about it. Section~\ref{sec4} concludes the paper.

\section{Preliminaries}\label{sec2}

Throughout the paper, we always consider finite abelian groups $G$. Let $A_1$ and $A_2$ be two subsets of a group $G$. For each $y \in G$, define the \emph{weight enumerator} of $A_1$ and $A_2$ at $y$ as
$$
\nu_{A_1,A_2}(y)=|\{(a_1,a_2) \in A_1 \times A_2 \mid y=a_1a_2^{-1}\}|.
$$
When $A_1=A_2$, we simply write $\nu_{A_1,A_2}(y)$ as $\nu_{A_1}(y)$.

We use $\Z[G]$ to denote the group ring. For $A \in \Z[G]$ with  nonnegative coefficients, we use $\{A\}$ to denote the underlying subset of $G$ corresponding to the elements of $A$ with positive coefficients and $[A]$ the multiset corresponding to $A$. For a subset $B$ of $G$, the inclusion $B \subset [A]$ means each element of $B$ occurs at least once in the multiset $[A]$. For $A \in \Z[G]$ and $g \in G$, we use $[A]_g$ to denote the coefficient of $g$ in $A$. Suppose $A=\sum_{g \in G} a_gg \in \Z[G]$, then $A^{(-1)}$ is defined to be $\sum_{g \in G} a_gg^{-1}$. Suppose $A=\sum_{g \in G} a_gg \in \Z[G]$ and $B=\sum_{g \in G} b_gg \in \Z[G]$, then the product $AB$ is defined to be $\sum_{g \in G} (\sum_{h \in G}a_{gh^{-1}}b_h) g$. A \emph{character} $\chi$ of $G$ is a group homomorphism from $G$ to the multiplicative group of the complex field $\mathbb{C}$. For a group $G$, we use $\wh{G}$ to denote its character group. There exists a group isomorphism $\De:G \rightarrow \wh{G}$, such that for each $y \in G$, we have $\chi_y:=\De(y) \in \wh{G}$. Therefore, $\wh{G}=\{\chi_y \mid y \in G\}$. For $\chi \in \wh{G}$ and $A \in \Z[G]$, we use $\chi(A)$ to denote the character sum $\sum_{x \in A} \chi(x)$. For a more detailed treatment of group rings and characters, please refer to \cite[Chapter 1]{Pott95}.

Now we are ready to define formally dual pairs.

\begin{definition}[Formally dual pair]\label{def-iso}
Let $\De$ be a group isomorphism from $G$ to $\wh{G}$, such that $\De(y)=\chi_y$ for each $y \in G$. Let $S$ and $T$ be subsets of $G$. Then $S$ and $T$ form a formally dual pair in $G$ under the isomorphism $\De$, if for each $y \in G$,
\begin{equation}\label{eqn-def}
|\chi_y(S)|^2=\frac{|S|^2}{|T|}\nu_T(y).
\end{equation}
\end{definition}

\begin{remark}\label{rem-def}
\quad
\begin{itemize}
\item[(1)] According to \cite[Remark 2.10]{CKRS}, the roles of the two subsets $S$ and $T$ in a formally dual pair are interchangeable, in the sense that \eqref{eqn-def} holds for each $y \in G$, if and only if
    \begin{equation}\label{eqn-def2}
      |\chi_y(T)|^2=\frac{|T|^2}{|S|}\nu_S(y)
    \end{equation}
    holds for each $y \in G$.
\item[(2)] By Definition~\ref{def-iso}, formal duality depends only on $SS^{(-1)}$ and $TT^{(-1)}$. For each $g_1, g_2\in G$, suppose that $S^{\pr}=\{g_1x \mid x \in S\}$ is a translation of $S$ and $T^{\pr}=\{g_2x \mid x \in T\}$ is a translation of $T$. Then $S^{\pr}$ and $T^{\pr}$ also form a formally dual pair in $G$. Hence, formal duality is invariant under translation.
\item[(3)] By \cite[Proposition 2.9]{LPS}, we know that $S$ and $T$ form a formally dual pair in $G$ under the isomorphism $\De_1$ if and only if $S$ and $\De_2^{-1}(\De_1(T))$ form a formally dual pair in $G$ under the isomorphism $\De_2$. Thus, Definition~\ref{def-iso} does not depend on the specific choice of $\De$. From now on, by referring to a formally dual pair, we always assume a proper group isomorphism is chosen. In our concrete constructions below, we always use a group isomorphism $\De: G \rightarrow \wh{G}$, such that $\De(y)=\chi_y$ for each $y \in G$. Therefore, once we specify how the character $\chi_y$ is defined, the group isomorphism $\De$ follows immediately.
\item[(4)] By \cite[Theorem 2.8]{CKRS}, we must have $|G|=|S|\cdot|T|$. Hence, a formally dual pair in a group of nonsquare order, must contain two subsets with unequal sizes.
\end{itemize}
\end{remark}

To exclude some trivial examples of formally dual pairs, the concept of primitive formally dual pair was proposed in \cite[p. 134]{CKRS}.

\begin{definition}[Primitive formally dual pair]\label{def-primi}
For a subset $S$ of a group $G$, define $S$ to be a primitive subset of $G$, if $S$ is not contained in a coset of a proper subgroup of $G$ and $S$ is not a union of cosets of a nontrivial subgroup in $G$. For a formally dual pair $S$ and $T$ in $G$, it is a primitive formally dual pair, if both $S$ and $T$ are primitive subsets.
\end{definition}

A subset $S \subset G$ is called a (primitive) formally dual set in $G$, if there exists a subset $T \subset G$, such that $S$ and $T$ form a (primitive) formally dual pair in $G$. The following definition concerns the equivalence of formally dual pairs \cite[Definition 2.17]{LPS}. Given a group $G$, we use $\Aut(G)$ to denote its automorphism group.

\begin{definition}[Equivalence of formally dual pair]\label{def-equiv}
Let $S$ and $S^\pr$ be two formally dual sets in $G$. They are equivalent if there exist $g \in G$ and $\phi \in \Aut(G)$, such that
$$
S^\pr=g\phi(S).
$$
Moreover, let $S,T$ and $S^{\pr},T^{\pr}$ be two formally dual pairs in $G$. They are equivalent if one of $S$ and $T$ is equivalent to one of $S^{\pr}$ and $T^{\pr}$.
\end{definition}

As noted in Definition~\ref{def-equiv}, the equivalence of formally dual pairs can be reduced to the equivalence of formally dual sets. For $A \in \Z[G]$, the multiset
$$
[[AA^{(-1)}]_g \mid g \in G]
$$
is called the \emph{difference spectrum} of $A$. The multiset
$$
[|\chi(A)|^2 \mid \chi \in \wh{G}]
$$
is called the \emph{character spectrum} of $A$. The difference spectrum and character spectrum contain very detailed information about the formally dual pairs. Indeed, both of them are invariants with respect to the equivalence of formally dual sets.

Next, we mention a very powerful product construction.

\begin{proposition}[Product construction]{\rm \cite[Proposition 2.7]{LP}}\label{prop-prod}
Let $S_1$ and $T_1$ be a primitive formally dual pair in $G_1$. Let $S_2$ and $T_2$ be a primitive formally dual pair in $G_2$. Then $S_1 \times S_2$ and $T_1 \times T_2$ form a primitive formally dual pair in $G_1 \times G_2$.
\end{proposition}

Finally, we give a brief account of a lifting construction framework raised in \cite[Section 3]{LP}, which generates new primitive formally dual pairs from known ones. It is worthy noting that this lifting construction framework led to the first infinite family of primitive formally dual pairs which are formed by two subsets having unequal sizes \cite[Theorem 4.2]{LP}.

Let $G$ be a group of square order. Let $S$ and $T$ be a primitive formally dual pair in $G$ under the isomorphism $\De$, with $\De(y)=\chi_y$ for each $y \in G$. Suppose $|S|=|T|=\sqrt{|G|}$ and $S$ can be partitioned into two subsets $S_0$ and $S_1$. Define two subsets $S^\pr, T^\pr \subset \Z_2 \times G$ as follows:
\begin{align}
\begin{aligned}\label{eqn-lifting}
S^\pr&=\{(0,x) \mid x \in S_0\} \cup \{(1,x) \mid x \in S_1\}, \\
T^\pr&=\{(0,x) \mid x \in T\} \cup \{(1,x) \mid x \in T^{(-1)}\}.
\end{aligned}
\end{align}
Clearly, $|S^\pr|=\sqrt{|G|}$ and $|T^\pr|=2\sqrt{|G|}$.

Equation \eqref{eqn-lifting} describes a lifting construction framework so that we can use a primitive formally dual pair $S$ and $T$ in $G$ with $|S|=|T|$ as a starter, and generate a new formally dual pair $S^\pr$ and $T^\pr$ in $\Z_2 \times G$ with $|S^\pr|\ne|T^\pr|$. Indeed, a necessary and sufficient condition ensuring that $S^\pr$ and $T^\pr$ form a formally dual pair in $\Z_2 \times G$ is known.

\begin{proposition}{\rm\cite[Corollary 3.4]{LP}}\label{prop-lifting}
Let $S^{\pr}$ and $T^{\pr}$ be the subsets defined in \eqref{eqn-lifting}. Then $S^{\pr}$ and $T^{\pr}$ form a primitive formally dual pair in $\Z_2 \times G$ if and only if
\begin{equation*}
|\chi_z(T+T^{(-1)})|^2=\frac{4|T|^2}{|S|}(\nu_{S_0}(z)+\nu_{S_1}(z)), \quad \mbox{for each $z \in G$}.
\end{equation*}
\end{proposition}

\begin{remark}\label{rem-par}
To apply the lifting construction framework \eqref{eqn-lifting}, we need to deal with the following two crucial points:
\begin{itemize}
\item[(1)] Choose a proper initial primitive formally dual pair $S$ and $T$ in a group $G$, satisfying $|S|=|T|$.
\item[(2)] Find a proper partition of $S$ into $S_0$ and $S_1$.
\end{itemize}
\end{remark}

In the next section, we will employ the lifting construction framework \eqref{eqn-lifting} to construct primitive formally dual pairs in $\Z_2 \times \Z_4^{2m}$.

\section{A direct construction of primitive formally dual pairs in $\Z_2 \times \Z_4^{2m}$}\label{sec3}

In this section, we propose a direct construction to generate an infinite family of primitive formally dual pairs in $\Z_2 \times \Z_4^{2m}$, where the two subsets have unequal sizes. This family has been discovered in \cite[Theorem 6.2]{LP} using a recursive approach. We remark that the direct construction offers more insights to this infinite family.

Now we introduce some notation which will be used throughout the rest of this paper. First, we define the canonical characters on $\Z_4^n$ and $\Z_2 \times \Z_4^n$, which will be used later. For each $w \in \Z_2$,  the character $\var_w \in \wh{\Z_2}$ is defined as $\var_w(a)=(-1)^{wa}$ for each $a \in \Z_2$. For each $z=(z_1,z_2,\ldots,z_n) \in \Z_4^{n}$, define the character $\chi_z \in \wh{\Z_4^n}$ as $\chi_z(b)=(\sqrt{-1})^{z\cdot b}$ for each $b=(b_1,b_2,\ldots,b_n) \in \Z_4^{n}$, where $z\cdot b$ is defined as $\sum_{i=1}^n z_ib_i$. For each $(w,z) \in \Z_2 \times \Z_4^n$, define the character $\phi_{w,z} \in \wh{\Z_2 \times \Z_4^{n}}$ as $\phi_{w,z}((a,b))=\var_w(a)\chi_z(b)$ for each $(a,b) \in \Z_2 \times \Z_4^{n}$. Given a collection of sets $A_i$, $1 \le i \le l$, we use $\prod_{i=1}^l A_i$ to denote the Cartesian product of $A_i$'s.

We write a multiset as $[A]=[a_i \lan z_i \ran \mid 1 \le i \le t]$, which means for each $1 \le i \le t$, the element $a_i$ occurs $z_i$ times in $[A]$. For two nonnegative integers $a$ and $b$, we use $\binom{a}{b}$ to denote the usual binomial coefficient, namely,
$$
\binom{a}{b}=\begin{cases}
\frac{\prod_{i=0}^{b-1} (a-i)}{b!} & \mbox{if $b \le a$,} \\
0, & \mbox{if $b>a$.}
\end{cases}
$$

In order to describe our construction, we need more notation. Define
\begin{align*}
L&=\{(0,0),(0,1),(1,0),(3,3)\} \subset \Z_4^2 \\
L_1&=\{(0,0),(0,1),(1,0)\} \subset \Z_4^2 \\
L_2&=\{(3,3)\} \subset \Z_4^2
\end{align*}
where $L_1$ and $L_2$ form a partition of $L$. For $0 \le i \le m$, define a subset $E_{m,i}$ of $\Z_4^{2m}$ as
$$
E_{m,i}=\sum_{\substack{|\{1 \le j \le m \mid N_j=L_1\}|=i \\ |\{1 \le j \le m \mid N_j=L_2\}|=m-i}} \prod_{j=1}^m N_j.
$$

The infinite family in the next theorem has been discovered in \cite[Theorem 6.2]{LP} using a recursive approach. Below, we give a direct construction employing the lifting construction framework \eqref{eqn-lifting}.

\begin{theorem}\label{thm-dircon2}
Let $S=T=\prod_{j=1}^m L$. Define
\begin{equation}\label{eqn-S0dircon2}
S_0=\begin{cases}
  \sum_{i=0}^{\frac{m-1}{2}} E_{m,2i+1} & \mbox{if $m$ is odd,} \\[5pt]
  \sum_{i=0}^{\frac{m}{2}} E_{m,2i} & \mbox{if $m$ is even,}
  \end{cases}
\end{equation}
and
\begin{equation}\label{eqn-S1dircon2}
S_1=\begin{cases}
  \sum_{i=0}^{\frac{m-1}{2}} E_{m,2i} & \mbox{if $m$ is odd,} \\[5pt]
  \sum_{i=0}^{\frac{m}{2}-1} E_{m,2i+1} & \mbox{if $m$ is even,}
  \end{cases}
\end{equation}
which form a partition of $S$. Let
\begin{equation}
\begin{aligned}\label{eqn-Tpdircon2}
S^{\pr}&=\{(0,x) \mid x \in S_0\} \cup \{(1,x) \mid x \in S_1\}, \\
T^{\pr}&=\{(0,x) \mid x \in T\} \cup \{(1,x) \mid x \in T^{(-1)}\}.
\end{aligned}
\end{equation}
Then $S^{\pr}$ and $T^{\pr}$ form a primitive formally dual pair in $\Z_2 \times \Z_4^{2m}$. Moreover, we have
\begin{align*}
&[[T^{\pr}T^{\pr(-1)}]_g \mid g \in \Z_2 \times \Z_4^{2m}] \\
=&[0\lan 2^{4m+1}-10^{m}-13^{m}\ran, 2^l\lan 12^{m-\frac{l-1}{2}}\binom{m}{\frac{l-1}{2}}+2^{2m-l+1}3^{l-1}\binom{m}{l-1}\ran \mid 1 \le l \le m+1, \mbox{$l$ odd}] \\
&\cup [2^l\lan 2^{2m-l+1}3^{l-1}\binom{m}{l-1}\ran \mid 2 \le l \le m+1, \mbox{$l$ even}] \\
&\cup [ 2^l\lan 12^{m-\frac{l-1}{2}}\binom{m}{\frac{l-1}{2}}\ran \mid m+2 \le l \le 2m+1, \mbox{$l$ odd}].
\end{align*}
\end{theorem}

\begin{remark}
In \cite[Theorem 6.2]{LP}, we can only derive the frequency of $0$ in the difference spectrum of $T^{\pr}$. The direct construction demonstrated below provides more insights into the structure of $T^{\pr}$, which enable us to compute the difference spectrum of $T^{\pr}$. In addition, we also know the character spectrum of $S^{\pr}$ by \eqref{eqn-def}.
\end{remark}


Note that $L$ and $L$ form a primitive formally dual pair in $\Z_4^2$ \cite[Theorem 3.7(1)]{LPS}. By Proposition~\ref{prop-prod}, $S=\prod_{j=1}^m L$ and $T=\prod_{j=1}^m L$ form a primitive formally dual pair in $\Z_4^{2m}$. Note that the construction in Theorem~\ref{thm-dircon2} fits into the lifting construction framework \eqref{eqn-lifting}. By Proposition~\ref{prop-lifting}, in order to show that $S^{\pr}$ and $T^{\pr}$ form a primitive formally dual pair, it suffices to show that
\begin{equation}\label{eqn-eqn8}
|\chi_z(T+T^{(-1)})|^2=\frac{4|T|^2}{|S|}(\nu_{S_0}(z)+\nu_{S_1}(z)), \quad \mbox{for each $z \in \Z_4^{2m}$}.
\end{equation}

Now we proceed to compute the left and right hand sides of \eqref{eqn-eqn8}. We first consider the right hand side. To understand $S_0S_0^{(-1)}$ and $S_1S_1^{(-1)}$, we need to compute $E_{m,i}E_{m,j}^{(-1)}$. For this purpose, we introduce more notation below. Define four subsets of $\Z_4^2$ as
\begin{align*}
N_{11}&=\{L_1L_1^{(-1)}\} \sm \{(0,0)\}, \\
N_{12}&=\{L_1L_2^{(-1)}\}, \\
N_{21}&=\{L_2L_1^{(-1)}\}, \\
N_{22}&=\{L_2L_2^{(-1)}\}=\{(0,0)\}.
\end{align*}
Define a subset of $\Z_4^{2m}$ as
\begin{equation}\label{eqn-N}
N=\underbrace{N_{11} \times \cdots \times N_{11}}_{w} \times \underbrace{N_{12} \times \cdots \times N_{12}}_{u-w} \times \underbrace{N_{21} \times \cdots \times N_{21}}_{v-w} \times \underbrace{N_{22} \times \cdots \times N_{22}}_{m-u-v+w}.
\end{equation}
Note that $\Z_4^2$ can be partitioned as
$$
\Z_4^2=N_{11} \cup N_{12} \cup N_{21} \cup N_{22} \cup Z,
$$
where $Z=\{(0,2),(2,0),(2,2)\}$. For $0 \le u,v,w \le 2m$ and $\max\{0,u+v-m\} \le w \le \min\{u,v\}$, define
$$
F_{m,u,v,w}=\sum_{\substack{ |\{ 1 \le j \le m \mid N_j=N_{11}\}|=w \\
|\{ 1 \le j \le m \mid N_j=N_{12} \}|=u-w \\
|\{ 1 \le j \le m \mid N_j=N_{21}\}\}|=v-w  \\
|\{ 1 \le j \le m \mid N_j=N_{22} \}|=m-u-v+w }} \prod_{j=1}^m N_j.
$$
Hereafter, when we write $F_{m,u,v,w}$, we always assume that $0 \le u,v,w \le 2m$ and $\max\{0,u+v-m\} \le w \le \min\{u,v\}$ hold. Define
$$
M_m=\{(z_1,z_2,\ldots,z_m) \in \Z_4^{2m} \mid \mbox{each $z_i \in \Z_4^2$ and there exists $z_j \in Z$} \}.
$$
Therefore, $\Z_4^{2m}$ can be partitioned as
$$
\Z_4^{2m}=(\bigcup_{\substack{ 0 \le u,v,w \le m \\ \max\{0,u+v-m\} \le w \le \min\{u,v\}}} F_{m,u,v,w})\bigcup M_m.
$$

Let $\sig \in \Sym(m)$. For $z=(z_1,z_2,\ldots,z_m) \in \Z_4^{2m}$, where $z_i \in \Z_4^2$, define
$$
\sig(z)=(z_{\sig(1)},z_{\sig(2)},\ldots,z_{\sig(m)}).
$$
The action of $\sig$ on elements of $\Z_4^{2m}$ can be naturally extended to a subset of $\Z_4^{2m}$. For instance, we have
\begin{align*}
\sig(E_{m,i})&=\sum_{\substack{|\{1 \le j \le m \mid N_j=L_1\}|=i \\ |\{1 \le j \le m \mid N_j=L_2\}|=m-i}} \prod_{j=1}^m N_{\sig(j)}=E_{m,i}, \\
\sig(F_{m,u,v,w})&=\sum_{\substack{ |\{ 1 \le j \le m \mid N_j=N_{11}\}|=w \\
|\{ 1 \le j \le m \mid N_j=N_{12} \}|=u-w \\
|\{ 1 \le j \le m \mid N_j=N_{21}\}\}|=v-w  \\
|\{ 1 \le j \le m \mid N_j=N_{22} \}|=m-u-v+w }} \prod_{j=1}^m N_{\sig(j)}=F_{m,u,v,w}.
\end{align*}
By the definitions of $N$ and $F_{m,u,v,w}$, we have $F_{m,u,v,w}=\bigcup_{\sig \in \Sym(m)} \{\sig(N)\}$.

The following lemma concerns $F_{m,u,v,w}$, as well as the relation between $F_{m,u,v,w}$ and $E_{m,i}E_{m,j}^{(-1)}$.

\begin{lemma}\label{lem-FE}
\begin{itemize}
\item[(1)] $F_{m,u,v,w} \subset [E_{m,i}E_{m,j}^{(-1)}]$ if and only if $i=u+h$ and $j=v+h$ for some $0 \le h \le m-u-v+w$.
\item[(2)] For each $x \in F_{m,u,v,w}$ and $0 \le h \le m-u-v+w$, we have $[E_{m,u+h}E_{m,v+h}^{(-1)}]_x=3^h\binom{m-u-v+w}{h}$.
\end{itemize}
\end{lemma}
\begin{proof}
(1) Suppose $i=u+h$ and $j=v+h$ for some $0 \le h \le m-u-v+w$. Consider
$$
P_1 = \underbrace{L_1 \times \cdots \times L_1}_{w} \times \underbrace{L_1 \times \cdots \times L_1}_{u-w} \times \underbrace{L_2 \times \cdots \times L_2}_{v-w} \times \underbrace{L_1 \times \cdots \times L_1}_{h} \times \underbrace{L_2 \times \cdots \times L_2}_{m-u-v+w-h}
$$
and
$$
P_2 = \underbrace{L_1 \times \cdots \times L_1}_{w} \times \underbrace{L_2 \times \cdots \times L_2}_{u-w} \times \underbrace{L_1 \times \cdots \times L_1}_{v-w} \times \underbrace{L_1 \times \cdots \times L_1}_{h} \times \underbrace{L_2 \times \cdots \times L_2}_{m-u-v+w-h}.
$$
Note that $P_1 \subset [E_{m,u+h}]$, $P_2 \subset [E_{m,v+h}]$ and $N \subset [P_1P_2^{(-1)}] \subset [E_{m,u+h}E_{m,v+h}^{(-1)}]$, where $N$ is defined in \eqref{eqn-N}. For any $\sig \in \Sym(m)$, we have $\sig(N) \subset [\sig(P_1)\sig(P_2)^{(-1)}] \subset [\sig(E_{m,u+h})\sig(E_{m,v+h}^{(-1)})]= [E_{m,u+h}E_{m,v+h}^{(-1)}]$. Since $F_{m,u,v,w}=\bigcup_{\sig \in \Sym(m)} \{\sig(N)\}$, we have
$$
F_{m,u,v,w} \subset [E_{m,u+h}E_{m,v+h}^{(-1)}].
$$
Conversely, suppose $F_{m,u,v,w} \subset [E_{m,i}E_{m,j}^{(-1)}]$, then $N \subset [E_{m,i}E_{m,j}^{(-1)}]$. By the definition of $E_{m,i}$ and $E_{m,j}$, there exist $P_1^{\pr} \subset E_{m,i}$ and $P_2^{\pr} \subset E_{m,j}$, such that $P_1^{\pr}$ and $P_2^{\pr}$ are formed by products of $L_1$ and $L_2$, and $N \subset [P_1^{\pr}P_2^{\pr(-1)}]$. Since $N_{22}=\{(0,0)\}$ belongs to $\{L_1L_1^{(-1)}\}$ and $\{L_2L_2^{(-1)}\}$, this forces
\begin{align*}
P_1^{\pr}&=\underbrace{L_1 \times \cdots \times L_1}_{w} \times \underbrace{L_1 \times \cdots \times L_1}_{u-w} \times \underbrace{L_2 \times \cdots \times L_2}_{v-w} \times \underbrace{\star \times \cdots \times \star}_{m-u-v+w}, \\
P_2^{\pr}&=\underbrace{L_1 \times \cdots \times L_1}_{w} \times \underbrace{L_2 \times \cdots \times L_2}_{u-w} \times \underbrace{L_1 \times \cdots \times L_1}_{v-w} \times \underbrace{\star \times \cdots \times \star}_{m-u-v+w},
\end{align*}
where for each of the last $m-u-v+w$ components, the two subsets in $P_1^{\pr}$ and $P_2^{\pr}$ are either both $L_1$ or both $L_2$. Suppose for some $0 \le h \le m-u-v+w$, exactly $h$ of the last $m-u-v+w$ components in $P_1^{\pr}$ and $P_2^{\pr}$ contain both $L_1$. Then,
we have $i=u+h$ and $j=v+h$.

(2) By the definition of $N$, $E_{m,u+h}$, $E_{m,v+h}$ and $F_{m,u,v,w}=\bigcup_{\sig \in \Sym(m)} \{\sig(N)\}$, we can see that
\begin{itemize}
\item[a)] For all $x,y \in N$, we have $[E_{m,u+h}E_{m,v+h}^{(-1)}]_x=[E_{m,u+h}E_{m,v+h}^{(-1)}]_y$.
\item[b)] For each $y \in F_{m,u,v,w}$, there exists $x \in N$ and $\sig \in \Sym(m)$, such that $y=\sig(x)$.
\end{itemize}
Combining a) and b), we conclude that $[E_{m,u+h}E_{m,v+h}^{(-1)}]_x=[E_{m,u+h}E_{m,v+h}^{(-1)}]_y$ for all $x,y \in F_{m,u,v,w}$. Without loss of generality, we assume that $x \in N$. Then there exists $Q_1 \subset E_{m,u+h}$ and $Q_2 \subset E_{m,v+h}$ where $Q_1$ and $Q_2$ are products of $L_1$ and $L_2$, such that $x \in [Q_1Q_2^{(-1)}]$. This forces
\begin{align*}
Q_1&=\underbrace{L_1 \times \cdots \times L_1}_{w} \times \underbrace{L_1 \times \cdots \times L_1}_{u-w} \times \underbrace{L_2 \times \cdots \times L_2}_{v-w} \times \underbrace{\star \times \cdots \times \star}_{m-u-v+w}, \\
Q_2&=\underbrace{L_1 \times \cdots \times L_1}_{w} \times \underbrace{L_2 \times \cdots \times L_2}_{u-w} \times \underbrace{L_1 \times \cdots \times L_1}_{v-w} \times \underbrace{\star \times \cdots \times \star}_{m-u-v+w},
\end{align*}
where for each of the last $m-u-v+w$ components, the two subsets in $Q_1$ and $Q_2$ are either both $L_1$ or both $L_2$, and there are exactly $h$ components containing both $L_1$, where $0 \le h \le m-u-v+w$. Hence, there are $\binom{m-u-v+w}{h}$ ways to choose $h$ components containing both $L_1$.  Notice that $[L_1L_1^{(-1)}]_{(0,0)}=3$, in each of these $h$ components containing both $L_1$, there are three distinct ways to express $(0,0)$ as a difference of elements from $L_1$. Similarly, since $[L_2L_2^{(-1)}]_{(0,0)}=1$, in the remaining $m-u-v+w-h$ components containing both $L_2$, there is a unique way to express $(0,0)$ as a difference of elements from $L_2$. Thus, we have $[Q_1Q_2^{(-1)}]_x=3^h$. In total, we get $[E_{m,u+h}E_{m,v+h}^{(-1)}]_x=3^h\binom{m-u-v+w}{h}$.
\end{proof}

Employing Lemma~\ref{lem-FE}, we can determine the multiset $[S_0S_0^{(-1)}+S_1S_1^{(-1)}]$.

\begin{proposition}\label{prop-S0S0invS1S1inv2}
Let $S_0$ and $S_1$ be the two subsets defined in \eqref{eqn-S0dircon2} and \eqref{eqn-S1dircon2}, respectively. For $z \in \Z_4^{2m}$, we have
$$
[S_0S_0^{(-1)}+S_1S_1^{(-1)}]_z=\begin{cases}
  0 & \mbox{if $z \in M_{m}$,} \\
  0 & \mbox{if $z \in F_{m,u,v,w}$ and $u+v$ odd,} \\
  4^{m-u-v+w} & \mbox{if $z \in F_{m,u,v,w}$ and $u+v$ even.}
\end{cases}
$$
\end{proposition}
\begin{proof}
We only prove the case of $m$ being odd. The proof of $m$ even case is completely analogous. Recalling that $S_0=\sum_{i=0}^{\frac{m-1}{2}} E_{m,2i+1}$ and $S_1=\sum_{i=0}^{\frac{m-1}{2}} E_{m,2i}$, we have
$$
S_0S_0^{(-1)}=\sum_{0 \le i,j \le \frac{m-1}{2}} E_{m,2i+1}E_{m,2j+1}^{(-1)}, \quad S_1S_1^{(-1)}=\sum_{0 \le i,j \le \frac{m-1}{2}} E_{m,2i}E_{m,2j}^{(-1)}.
$$
By definition, $[S_0S_0^{(-1)}+S_1S_1^{(-1)}]_z=0$ for each $z \in M_{m}$. Now let $z \in F_{m,u,v,w}$. By Lemma~\ref{lem-FE}(1), we know that
\begin{align*}
z \in [S_0S_0^{(-1)}] \Leftrightarrow &\quad \mbox{there exists some $0 \le i,j \le \frac{m-1}{2}$ and $0 \le h \le m-u-v+w$,} \\
                                        &\quad \mbox{such that $2i+1=u+h$, $2j+1=v+h$,}
\end{align*}
and
\begin{align*}
z \in [S_1S_1^{(-1)}] \Leftrightarrow &\quad \mbox{there exists some $0 \le i,j \le \frac{m-1}{2}$ and $0 \le h \le m-u-v+w$,} \\
                                        &\quad \mbox{such that $2i=u+h$, $2j=v+h$.}
\end{align*}
Therefore, if $u+v$ is odd, then $[S_0S_0^{(-1)}+S_1S_1^{(-1)}]_z=0$. If $u+v$ is even, then we have
$$
[S_0S_0^{(-1)}]_z=\begin{cases}
  \sum_{\substack{0 \le h \le m-u-v+w \\ h \equiv 0 \bmod2}} [E_{m,u+h}E_{m,v+h}^{(-1)}]_z & \mbox{if $u$ and $v$ both odd,} \\
  \sum_{\substack{0 \le h \le m-u-v+w \\ h \equiv 1 \bmod2}} [E_{m,u+h}E_{m,v+h}^{(-1)}]_z & \mbox{if $u$ and $v$ both even,} \\
\end{cases}
$$
and
$$
[S_1S_1^{(-1)}]_z=\begin{cases}
  \sum_{\substack{0 \le h \le m-u-v+w \\ h \equiv 1 \bmod2}} [E_{m,u+h}E_{m,v+h}^{(-1)}]_z & \mbox{if $u$ and $v$ both odd,} \\
  \sum_{\substack{0 \le h \le m-u-v+w \\  h \equiv 0 \bmod2}} [E_{m,u+h}E_{m,v+h}^{(-1)}]_z & \mbox{if $u$ and $v$ both even.} \\
\end{cases}
$$
Together with Lemma~\ref{lem-FE}(2), we have
\begin{align*}
[S_0S_0^{(-1)}+S_1S_1^{(-1)}]_z&=\sum_{0 \le h \le m-u-v+w} [E_{m,u+h}E_{m,v+h}^{(-1)}]_z \\
                               &=\sum_{0 \le h \le m-u-v+w} 3^h\binom{m-u-v+w}{h}=4^{m-u-v+w},
\end{align*}
which completes the proof.
\end{proof}

Next, we compute the left hand side of \eqref{eqn-eqn8} in the following proposition.

\begin{proposition}\label{prop-TTinv2}
Let $T=\prod_{j=1}^m L$. For $z \in \Z_4^{2m}$, we have
$$
|\chi_z(T+T^{(-1)})|^2=\begin{cases}
  0 & \mbox{if $z \in M_{m}$,} \\
  0 & \mbox{if $z \in F_{m,u,v,w}$ and \mbox{$u+v$ odd},} \\
  4^{2m-u-v+w+1} & \mbox{if $z \in F_{m,u,v,w}$ and \mbox{$u+v$ even}.}
\end{cases}
$$
\end{proposition}
\begin{proof}
Recall that $L=\{(0,0), (0,1), (1,0), (3,3)\} \subset \Z_4^2$ and $Z=\{(0,2), (2,0), (2,2)\} \subset \Z_4^2$. For $y \in \Z_4^2$, it is easy to verify that
\begin{equation}\label{eqn-eqn9}
\chi_y(LL^{(-1)})=\begin{cases}
  16 & \mbox{if $y=(0,0)$,} \\
  0 & \mbox{if $y \in Z$,} \\
  4 & \mbox{if $y \in \Z_4^2 \sm (\{(0,0)\} \cup Z)$,}
\end{cases}
\end{equation}
where the subset $\Z_4^2 \sm (\{(0,0)\} \cup Z)=\{ L_1L_1^{(-1)}+L_1L_2^{(-1)}+L_2L_1^{(-1)} \} \sm \{(0,0)\}$, and
\begin{equation}\label{eqn-eqn10}
\chi_y(LL)=\chi_y(L^{(-1)}L^{(-1)})=\begin{cases}
  16 & \mbox{if $y \in \{(0,0)\}$,} \\
  0 & \mbox{if $y \in Z$,} \\
  4 & \mbox{if $y \in \{ (0,1),(0,3),(1,0),(3,0),(1,3),(3,1) \}$,} \\
  -4 & \mbox{if $y \in \{ (1,1),(3,3),(1,2),(3,2),(2,1),(2,3) \}$,}
  \end{cases}
\end{equation}
where the subsets
\begin{align*}
\{ (0,1),(0,3),(1,0),(3,0),(1,3),(3,1) \}&=\{ L_1L_1^{(-1)} \} \sm \{(0,0)\}, \\
\{ (1,1),(3,3),(1,2),(3,2),(2,1),(2,3) \}&=\{L_1L_2^{(-1)}+L_2L_1^{(-1)}\}.
\end{align*}

It is straightforward to verify that $LL=L^{(-1)}L^{(-1)}$. Since $T=\prod_{j=1}^m L$, we have $TT^{(-1)}=(LL^{(-1)})^m$ and $TT=(LL)^m=(L^{(-1)}L^{(-1)})^m=T^{(-1)}T^{(-1)}$. Consequently,
$$
(T+T^{(-1)})(T+T^{(-1)})=2(TT^{(-1)}+TT).
$$
Therefore, for $z=(z_1,z_2,\ldots,z_m) \in \Z_4^{2m}$, in which $z_i \in \Z_4^2$, $1 \le i \le m$, we have
$$
|\chi_z(T+T^{(-1)})|^2=2(\chi_z(TT^{(-1)})+\chi_z(TT))=2(\prod_{i=1}^m\chi_{z_i}(LL^{(-1)})+\prod_{i=1}^m\chi_{z_i}(LL)).
$$
If $z \in M_{m}$, then by \eqref{eqn-eqn9} and \eqref{eqn-eqn10}, we have
\begin{equation}\label{eqn-eqn11}
|\chi_z(T+T^{(-1)})|^2=0, \quad \mbox{if $z \in M_{m}$}.
\end{equation}
If $z \notin M_{m}$, then $z \in F_{m,u,v,w}$ for some $0 \le u,v,w \le m$. Therefore, we have
\begin{align*}
|\{1 \le i \le m \mid z_i \in \{L_1L_1^{(-1)}\} \sm \{(0,0)\}\}|&=w, \\
|\{1 \le i \le m \mid z_i \in \{L_1L_2^{(-1)}+L_2L_1^{(-1)}\}\}|&=u+v-2w, \\
|\{1 \le i \le m \mid z_i \in \{L_2L_2^{(-1)}\}\}|&=m-u-v+w.
\end{align*}
Together with \eqref{eqn-eqn9} and \eqref{eqn-eqn10}, for $z \in F_{m,u,v,w}$, we have
\begin{align}\label{eqn-eqn12}
\begin{aligned}
|\chi_z(T+T^{(-1)})|^2&=2(\prod_{i=1}^m\chi_{z_i}(LL^{(-1)})+\prod_{i=1}^m\chi_{z_i}(LL)) \\
                      &=2(16^{m-u-v+w}4^{u+v-w}+16^{m-u-v+w}4^w(-4)^{u+v-2w}) \\
                      &=2\cdot4^{2m-u-v+w}(1+(-1)^{u+v})\\
                      &=\begin{cases}
                        0 & \mbox{if $z \in F_{m,u,v,w}$ and $u+v$ odd,} \\
                        4^{2m-u-v+w+1} & \mbox{if $z \in F_{m,u,v,w}$ and $u+v$ even.}
                      \end{cases}
\end{aligned}
\end{align}
Combining \eqref{eqn-eqn11} and \eqref{eqn-eqn12}, we complete the proof.
\end{proof}

In the following, we proceed to compute the multiset $[[T^{\pr}T^{\pr(-1)}]_g \mid g \in \Z_2 \times \Z_4^{2m}]$. Denote $Y=Z \cup \{(0,0)\}$ and define
$$
I=\{(0,1),(0,3),(1,0),(3,0),(1,1),(3,3)\} \subset \Z_4^2
$$
and
$$
O_m=\{(z_1,z_2,\ldots,z_m) \in \Z_4^{2m} \mid \mbox{each $z_i \in \Z_4^2$, and there exsits $z_j \in \Z_4^2 \sm (Y \cup I)$} \}.
$$

So far, we have proved several results containing structural information of the building blocks used in Theorem~\ref{thm-dircon2}, which is not known from \cite{LP}. Besids, the next lemma quotes a result of \cite{LP}, whose proof follows from a similar spirit as that of Proposition~\ref{prop-TTinv2}.

\begin{lemma}{\rm\cite[Lemma 6.4]{LP}}\label{lem-TTdiff2}
Let $T=\prod_{j=1}^m L$.
\begin{itemize}
\item[(1)] For $x=(x_1,x_2,\ldots,x_m) \in \Z_4^{2m}$, with $x_i \in \Z_4^2$,
$$
[TT^{(-1)}]_x=\begin{cases}
  0 & \mbox{if $x \in M_{m}$,} \\
  4^{l} & \mbox{if $x \notin M_{m}$ and $|\{1 \le i \le m \mid x_i \in \{(0,0)\}\}|=l$.}
\end{cases}
$$
\item[(2)] For $x=(x_1,x_2,\ldots,x_m) \in \Z_4^{2m}$, with $x_i \in \Z_4^2$,
$$
[TT]_x=[T^{(-1)}T^{(-1)}]_x=\begin{cases}
  0 & \mbox{if $x \in O_{m}$,} \\
  2^{l} & \mbox{if $x \notin O_{m}$ and $|\{1 \le i \le m \mid x_i \in I\}|=l$.}
\end{cases}
$$
\end{itemize}
\end{lemma}

Now we can compute the multiset $[[T^{\pr}T^{\pr(-1)}]_g \mid g \in \Z_2 \times \Z_4^{2m}]$.

\begin{proposition}\label{prop-TpTpinv2}
Let $T^{\pr}$ be the subset of $\Z_2 \times \Z_4^{2m}$ defined in \eqref{eqn-Tpdircon2}, then we have
\begin{align*}
&[[T^{\pr}T^{\pr(-1)}]_g \mid g \in \Z_2 \times \Z_4^{2m}] \\
=&[0\lan 2^{4m+1}-10^{m}-13^{m} \ran, 2^l\lan 12^{m-\frac{l-1}{2}}\binom{m}{\frac{l-1}{2}}+2^{2m-l+1}3^{l-1}\binom{m}{l-1} \ran \mid 1 \le l \le m+1, \mbox{$l$ odd}] \\
&\cup [2^l\lan 2^{2m-l+1}3^{l-1}\binom{m}{l-1} \ran \mid 2 \le l \le m+1, \mbox{$l$ even}] \\
&\cup [2^l\lan 12^{m-\frac{l-1}{2}}\binom{m}{\frac{l-1}{2}} \ran \mid m+2 \le l \le 2m+1, \mbox{$l$ odd}].
\end{align*}
\end{proposition}
\begin{proof}
Note that
\begin{equation}\label{eqn-eqn13}
T^{\pr}T^{\pr(-1)}=2\sum_{x \in [TT^{(-1)}]}(0,x)+\sum_{x \in [TT+T^{(-1)}T^{(-1)}]} (1,x)=2(\sum_{x \in [TT^{(-1)}]}(0,x)+\sum_{x \in [TT]}(1,x)),
\end{equation}
where $T=\prod_{j=1}^m L$. It suffices to determine the two multisets $[[TT^{(-1)}]_x \mid x \in \Z_4^{2m}]$ and $[[TT]_x \mid x \in \Z_4^{2m}]$.

By Lemma~\ref{lem-TTdiff2}(1), we have
\begin{equation}\label{eqn-eqn14}
[[TT^{(-1)}]_x \mid x \in \Z_4^{2m}]=[ 0 \lan 4^{2m}-13^{m} \ran, 4^l \lan 12^{m-l}\binom{m}{l}\ran \mid 0 \le l \le m ].
\end{equation}

According to Lemma~\ref{lem-TTdiff2}(2), we have
\begin{equation}\label{eqn-eqn15}
[[TT]_x \mid x \in \Z_4^{2m} ]=[ 0 \lan4^{2m}-10^{m}\ran, 2^l\lan 6^l4^{m-l}\binom{m}{l}\ran \mid 0 \le l \le m ].
\end{equation}
Combining \eqref{eqn-eqn13}, \eqref{eqn-eqn14} and \eqref{eqn-eqn15}, we complete the proof.
\end{proof}

Now we are ready to prove Theorem~\ref{thm-dircon2}.

\begin{proof}[Proof of Theorem~\ref{thm-dircon2}]
Applying Propositions~\ref{prop-lifting},~\ref{prop-S0S0invS1S1inv2} and~\ref{prop-TTinv2}, we derive that $S^{\pr}$ and $T^{\pr}$ form a primitive formally dual pair in $\Z_2 \times \Z_4^{2m}$. The multiset $[[T^{\pr}T^{\pr(-1)}]_g \mid g \in \Z_2 \times \Z_4^{2m}]$ follows from Proposition~\ref{prop-TpTpinv2}.
\end{proof}

\section{Conclusion}\label{sec4}

In this paper, we gave a direct construction of primitive formally dual pairs having subsets with unequal sizes in $\Z_2 \times \Z_4^{2m}$. While the derived infinite family had been discovered in \cite{LP} using a recursive approach, the new direct construction provided more detailed information about the primitive formally dual pairs. This advantage viewpoint leads to the difference spectrum of $T^{\pr}$ in Theorem~\ref{thm-dircon2}, which is not known before.

The formally dual pair indicates how one can form periodic configurations by taking the union of translations of a given lattice. In this sense, our constructions of formally dual pairs lead to schemes generating candidates of energy-minimizing periodic configurations.

Finally, we mention four open problems which seem to be interesting.

\begin{itemize}
\item[(1)] We remark that the two direct constructions in Theorem~\ref{thm-dircon2} and \cite[Theorem 4.2]{LP} both exploited the \emph{Teichmuller sets} in \emph{Galois rings}, whose additive group are of the form $\Z_4^n$. Thus, our construction suggests the possibility of more direct constructions involving Teichmuller sets.
\item[(2)] We think the general lifting construction framework \eqref{eqn-lifting} deserves further investigation. In particular, it is worthy noting that the lifting construction framework resembles the so called Waterloo decomposition of Singer difference sets \cite{ADJP}. So far, all known examples of primitive formally dual pairs having subsets with unequal sizes live in groups of the form $\Z_2 \times \Z_4^{2m}$, where $m \ge 1$. An interesting open problem is to construct such primitive formally dual pairs in other finite abelian groups.
\item[(3)] We note that for $N \le 1000$, there are only three open cases of primitive formally dual pairs in cyclic group $\Z_N$ \cite[Remark 5.12]{LPS}. In particular, the smallest open case in cyclic groups having unequal size subsets belongs to $\Z_{600}$, where the two subsets have size $10$ and $60$. We expect that advanced technique like the field descent method \cite{LS,Sch} can be exploited to improve the nonexistence results in cyclic groups.
\end{itemize}

\section*{Acknowledgement}

Shuxing Li is supported by the Alexander von Humboldt Foundation.

\end{document}